\documentclass[reqno,10pt]{amsart}
\usepackage{latexsym}
\usepackage[english]{babel}
\usepackage{xspace}
\usepackage[bookmarksnumbered,colorlinks]{hyperref}
\usepackage{graphics}
\usepackage{enumerate}
\usepackage{dsfont}
\newtheorem{theorem}{Theorem}[section]
\newtheorem{corollary}[theorem]{Corollary}
\newtheorem{lemma}[theorem]{Lemma}
\newtheorem{proposition}[theorem]{Proposition}
\theoremstyle{definition}
\newtheorem{definition}[theorem]{Definition}
\theoremstyle{remark}
\newtheorem{remark}[theorem]{Remark}
\newcommand{\bt}{\begin{theorem}}                               
\newcommand{\et}{\end{theorem}}                           
\newcommand{\bco}{\begin{corollary}}                           
\newcommand{\eco}{\end{corollary}}                             
\newcommand{\bd}{\begin{definition}}                        
\newcommand{\ed}{\end{definition}}                             
\newcommand{\bl}{\begin{lemma}}                            
\newcommand{\el}{\end{lemma}}                              
\newcommand{\bpr}{\begin{proposition}}                  
\newcommand{\epr}{\end{proposition}}                   
\newcommand{\bere}{\begin{remark}}              
\newcommand{\ere}{\end{remark}}                 
\newcommand{\beq}{\begin{equation}}
\newcommand{\eeq}{\end{equation}}
\def\bal#1\eal{\begin{align}#1\end{align}}                      
\def\baln#1\ealn{\begin{align*}#1\end{align*}}         
\def\bml#1\eml{\begin{multline}#1\end{multline}}       
\def\bmln#1\emln{\begin{multline*}#1\end{multline*}}  
\def\bga#1\ega{\begin{gather}#1\end{gather}}
\def\bgan#1\egan{\begin{gather*}#1\end{gather*}}
\newcommand{\de}{\mathrm{d}}                        
\newcommand{\N}{\ensuremath{\mathbb{N}}\xspace}     
\newcommand{\Z}{\ensuremath{\mathbb{Z}}\xspace}  
\newcommand{\Q}{\ensuremath{\mathbb{Q}}\xspace}
\newcommand{\R}{\ensuremath{\mathbb{R}}\xspace}     
\newcommand{\K}{\ensuremath{\mathbb{K}}\xspace}

\newcommand{\inte}{\int_0^1\!\!}

\newcommand{\df}{{\rm d}}
\newcommand{\dist}{{\rm d}}

\title[Infinitely many geodesics between two points]%
{A remark on the Morse Theorem about infinitely many geodesics between two points}

\author[E. Caponio]{Erasmo Caponio}
\address{Dipartimento di Matematica, \hfill\break\indent
Politecnico di Bari, Via Orabona 4,\hfill\break\indent
70125, Bari, Italy}
\email{caponio@poliba.it}

\author[M. A. Javaloyes]{Miguel \'Angel Javaloyes}
\address{Departamento de Matem\'aticas, \hfill\break\indent
Universidad de Murcia, Campus de Espinardo\hfill\break\indent
30100 Espinardo, Murcia, Spain}
\email{majava@um.es}

\thanks{This research is a result of the
activity developed within the Spanish-Italian Acci\'on Integrada HI2008.0106/Azione Integrata
Italia-Spagna IT09L719F1. \\
MAJ is partially supported by Regional J.
Andaluc\'{\i}a Grant P09-FQM-4496,  MICINN project MTM2009-10418, and
Fundaci\'{o}n S\'{e}neca project 04540/GERM/06. This research is a
result of the activity developed within the framework of the Programme
in Support of Excellence Groups of the Regi\'{o}n de Murcia, Spain, by
Fundaci\'{o}n S\'{e}neca, Regional Agency for Science and Technology
(Regional Plan for Science and Technology 2007-2010).}

\subjclass[2010]{53C20, 53C22, 57B30, 58B20}

\keywords{Geodesics, Morse theory, Finsler metrics.}

\begin{document}
\begin{abstract}
We show that any two non-conjugate points on a forward or backward complete connected Finsler manifold can be joined by infinitely many geodesics which are not covered by finitely many closed ones, provided that the Betti numbers of the based loop space  grow unbounded.
\end{abstract}

\maketitle
\section{Introduction}
A celebrated result by M. Morse \cite[Theorem 13.3, p. 239]{Morse96}  states that any two points on a complete connected Riemannian manifold $M$ can be joined by infinitely many geodesics provided that the homology $H_*(\Omega M, \K)$ of the based loop space of $M$, with respect to any coefficient field $\K$ is non trivial in infinitely many dimensions. 
By a theorem of J-P. Serre \cite{Serre51}, the assumption on the non triviality of the homology groups is satisfied if there exists $i>0$ such that $H_i(M,\K)$ is non trivial -- e.g. if $M$ is compact or non contractible. 

A related issue to the above is to find topological conditions under which the infinitely many geodesics expected from that result are ``geometrically distinct'', in the sense that they are not covered by finitely many closed geodesics of $M$,  as it doesn't happen, for example, for two non-antipodal points on a standard sphere which are connected by infinitely many ones whose  supports  are contained in the same closed geodesic.

Let $N(p,q,L)$ be the function which counts the number of geodesics between the points $p$ and $q$ having length strictly less than $L$. If the geodesics joining $p$ to $q$ are covered by a finite number $m$ of closed ones, then $N(p,q,L)\leq 2m (1+L/L_0)$, 
where $L_0$ is the length of the shortest prime closed geodesic. Thus,  superlinear growth of  $N(p,q,L)$ implies  that there exist infinitely many geometrically distinct geodesics.
For example, it is well known that in compact, connected, simply connected, rationally hyperbolic manifolds,  the integers $\mu_i=\sum_{j\leq i}\beta_j(\Omega M,\Q)$, $\beta_j=\dim(H_j(\Omega M,\Q))$, grow exponentially (see, for example, \cite[Prop. 5.6]{Patern99}) and then, if the points $p$ and $q$ are non-conjugate,  $N(p,q,L)$ has also exponential growth since, by a result of M. Gromov \cite{Gromov78}, there exists a constant  $C>0$ depending on $M$, such that
\beq N(p,q,L)\geq \sum_{j=1}^{C(L-1)}\beta_j(\Omega M,\K).\label{gromov}\eeq
On the other hand, for rationally elliptic manifolds,
the numbers $\mu_i$ grow at most polynomially. More generally, from the main theorem in \cite{McClea87} and \eqref{gromov}, it follows that, if $p$ and $q$ are non-conjugate, $N(p,q,L)$ has at least quadratic growth in any  simply connected, compact, connected Riemannian manifold that does not have the integral cohomology of a rank one symmetric space (see \cite[Corollary B]{McClea87}).

M. Tanaka in \cite[Problem C, p.183]{Tanaka82} asked if for a compact simply connected Riemannian manifold $(M,g)$ it is enough to assume  that the sequence of Betti numbers  
$\beta_i(\Omega M,\K)$ is unbounded to get that any two  points on $M$ can be joined by infinitely many geometrically distinct geodesics. He stated also that in  case the two point are non-conjugate the problem can be solved positively by an analogous technique as in the proof of Theorem 4.1 in \cite{Tanaka82}. 

In this note we give a detailed proof of the Tanaka's statement for 
a connected, complete, Riemannian manifold $(M,g)$. 
Rather than remaining in a Riemannian background, we will take advantage of the recently developed infinite-dimensional Morse theory for Finsler geodesics, \cite{CaJaMa09,CaJaMa11}, and will give  our result  in the larger class of Finsler metrics.
\bt\label{teo}
Let $(M,F)$ be a connected, forward or backward complete Finsler manifold and let $p$ and $q$ be two non-conjugate points of $M$. If the sequence of Betti numbers $\beta_i(\Omega M,\K)$, $i\in \N$, is unbounded, then there exist infinitely many geometrically distinct geodesics between $p$ and $q$.
\et
The assumption on the Betti numbers in Theorem~\ref{teo} is analogous to that on the Betti numbers of the free loop space in \cite{GroMey69a}. Actually the proof of Theorem~\ref{teo} is also based on an estimate on the Morse indices of iterated closed geodesics given in \cite{GroMey69a} and indeed we will not use \eqref{gromov}. We point out that when $M$ is a simply connected, compact, connected manifold that does not have the integral cohomology of a rank one symmetric space, then Theorem~\ref{teo} applies since, as shown in \cite{McClea87}, the Betti numbers $\beta_i(\Omega M,\Z_p)$ grow unbounded. 

We also point out that forward or backward completeness in Theorem \ref{teo} can be weakened into
condition \eqref{cap} below. This is because Palais-Smale condition and Morse theory for the energy functional of a Finsler metric  hold under that assumption  (see Remark 10 in \cite{CaJaMa09} and the comments before Theorem 5.2 in \cite{CaJaSa10}).

\section{Finsler metrics and Jacobi fields}
Let us introduce some definitions and properties about Finsler metrics. For more details,  see for example \cite{BaChSh00}.

Let $M$ be a  manifold, $TM$ its tangent bundle, $\pi:TM\rightarrow M$ the natural projection and $F:TM\rightarrow
[0,\infty)$ a continuous function. Then, we say that $F$ is a Finsler metric on $M$ if
\begin{enumerate}[(i)]
 \item $F$ is
smooth in $TM$
away from the zero section,
\item $F$ is positive homogeneous,
that is, $F(\lambda v)=\lambda v$ for every $v\in TM$ and
$\lambda>0$,
\item $F^2$ is fiberwise strongly convex, that is,
the {\it fundamental tensor} 
\[g_v(u,w):=\frac
12 \frac{\partial^2}{\partial t\partial s}
F^2(v+tu+sw)|_{t=s=0},\]
where $u,w\in T_{\pi(v)}M$, is positive definite for every $v\in TM$.
\end{enumerate}
Let us denote as $C_{p,q}$ the set of 
piecewise smooth curves from $p$ to $q$ defined on the interval $[0,1]$.  Then, we can define the {\em Finsler distance} from $p$ to $q$ 
as
\[{\dist}_F(p,q)=\inf_{\alpha\in C_{p,q}}\int_0^1 F(\dot \alpha(s))\df s\in [0,\infty),\]
where $\dot\alpha$ denotes the derivative of $\alpha$. Finsler metrics are not reversible in general and, as a consequence, the distance $\dist_F$ is not necessarily symmetric. Thus  we distinguish between forward and backward Cauchy sequences and forward and backward balls, denoted respectively by $B^+_F(x,r)$ and $B^-_F(x,r)$, for every $x\in M$ and $r>0$,
and defined as
\[B_F^+(p,r)=\{q\in M: \dist_F(p,q)<r\}, \quad
B^-_F(p,r)=\{q\in M:\dist_F(q,p)<r\}.\]
Moreover, we will denote by $\bar{B}^+_F(x,r)$ and $\bar{B}^-_F(x,r)$ their respective closures. 

We say that a Finsler metric satisfies condition $(\mathrm{cap})$ if 
\beq\label{cap}
\bar{B}^+_F(x,r)\cap \bar{B}^-_F(x,r)\text{ is compact,  for every $x\in M$ and $r>0$.}
\eeq
We remark that this condition is weaker than forward or backward completeness (see Example 4.6 in \cite{CaJaSa10}).

The {\it Cartan
tensor} is defined as the trilinear form
\[C_v(u_1,u_2,u_3)=\frac 14\left.
\frac{\partial^3}{\partial s_3\partial s_2\partial
s_1}F^2\left(v+\sum_{i=1}^3s_i u_i \right)\right|_{s_1=s_2=s_3=0}.\] Observe that $C_v$ is
symmetric, that is, the value does not depend on the order of
$u_1$, $u_2$ and $u_3$ and it is homogeneous of degree $-1$, that is, $C_{\lambda v}=\frac{1}{\lambda}C_v$. Moreover, if one of the
vectors $u_1$, $u_2$ and $u_3$ is proportional to $v$, then $C_v(u_1,u_2,u_3)=0$.
Given a vector field $V$ in an open set $U\subset M$ non-zero everywhere, we 
define the  connection $\nabla^V$ as the only linear connection on $U$ satisfying
\begin{eqnarray}
&\nabla^V_XY-\nabla^V_YX=[X,Y],\label{torsionfree} \\
&X( g_V(Y,Z))=g_V(\nabla^V_XY,Z)+g_V(Y,\nabla^V_XZ)+2 C_V(\nabla^V_XV,Y,Z),\label{almostg}
\end{eqnarray}
for any vector fields  $X$, $Y$ and $Z$ on $U$.  We observe that this connection can be in some way identified with the Chern connection (see \cite{radMathAnn04}).  Furthermore, we define  the curvature tensor of $g_V$  as
\[R^V(X,Y)Z=\nabla^V_X\nabla^V_YZ-\nabla^V_Y\nabla^V_XZ-
\nabla^V_{[X,Y]}Z.\] 
Observe that the connection $\nabla^V$ and, as a consequence, its curvature tensor $R^V$ are homogeneous of degree 0 in $V$.
For a smooth curve $\alpha:[a,b]\subseteq\R\rightarrow M$, and a vector field $Y$ along $\alpha$ (non zero everywhere),  the  above connection induces a covariant derivative, which will be denoted as $D_\alpha^YX$, for every 
vector field $X$ along $\alpha$. Indeed, this covariant derivative is determined by the following property: if $\bar{X}$, $\bar{Y}$ and $\bar{T}$ are vector fields extending $X$, $Y$ and $\dot \alpha$, then $\nabla^{\bar{Y}}_{\bar{T}}\bar{X}=D^Y_\alpha X$, that is $\nabla^{\bar{Y}}_{\bar{T}}\bar{X}$ does not depend on the extensions. 

On the other hand,
if $T=\dot\alpha$,  $R^T(V,T)T$ can be  defined by using
the covariant derivative. Indeed, choose a variation of
$\Psi:[-\varepsilon,\varepsilon]\times [a,b]\rightarrow M$ of
$\alpha$ with variation vector field $V$ and define
$\beta_t:[-\varepsilon,\varepsilon]\rightarrow M$ for $t\in [a,b]$
as $\beta_t(w)=\Psi(w,t)$ and $\gamma_w:[a,b]\rightarrow M$ as
$\gamma_w(t)=\Psi(w,t)$ for $w\in [-\varepsilon,\varepsilon]$.
Then
\[R^T(V,T)T:=\left.\left(D^{T_w}_{\beta_t}D^{T_w}_{\gamma_w}T_w-D^{T_w}_{\gamma_w}D^{T_w}_{\beta_t}T_w\right)\right|_{w=0},\]
where $T_w=\dot\gamma_w$ and $V_w=\dot \beta_t$. It can be shown that
this quantity does not depend  on the  variation $\Psi$.

Let $\Lambda M$ be the space of periodic curves $c\colon S^1\to M$ of  Sobolev class $H^1$ with respect to an auxiliary Riemannian metric $h$ of $M$. Moreover, given $p,q\in M$, let $\Omega_{pq}M$ be the space of $H^1$-curves from $p$ to $q$ parametrized on the interval $[0,1]$. 
 Periodic geodesics and geodesics connecting the points $p$ and $q$ in  $(M,F)$ are the critical points of the energy functional 
\[E_F(\alpha)=\frac 12\int_0^1 F^2(\dot\alpha) \df s,\]
defined respectively on the manifolds $\Lambda M$ and  $\Omega_{pq}M$. 
Thus using \eqref{torsionfree} and \eqref{almostg}, we can prove that they satisfy the equation
$D_\gamma^TT=0$, with $T=\dot\gamma$. Then the  equation of Jacobi fields of a geodesic $\gamma$, that is the equation satisfied by a variational vector field along $\gamma$ obtained as variation by geodesics, is given by $D^2_\gamma J=R^T(T,J)T$. 

\section{Morse index of iterated geodesics between two points}
The second variation of the energy functional $E_F$  at a geodesic $\gamma$ is given by
the index form
\[
I_\gamma (V,W)=\int_0^1 (g_T(V',W')+g_T(R^T(T,V)T, W)\de s,
\]
where $V$ and $W$ are $H^1$-vector fields along $\gamma$, $T=\dot\gamma$, $V'=D_\gamma^TV$ and $W'=D_\gamma^TW$. 

Let $c$ be a closed geodesic; the tangent space to $\Lambda M$ at $c$ is given by the $H^1$-vector fields $V$ along $c$ such that $V(0)=V(1)$, whereas the tangent space to
$\Omega_{pq}M$, at a geodesic $\gamma\colon [0,1]
\to M$ joining $p$ to $q$, consists of the $H^1$-vector fields along $\gamma$ such that $V(0)=0=V(1)$. 
The index form $I_c$ and $I_{\gamma}$ are defined respectively in  $T_{c}\Lambda M\times T_{c}\Lambda M$ and $T_{\gamma}\Omega_{pq}M\times T_{\gamma}\Omega_{pq}M$ and their indices are the dimensions of the maximal subspaces, in the respective tangent spaces, where they are negative definite.

In the following  we will also consider the situation when $p=q\in c(S^1)$ and consequently we will consider the index form $I_c$ as defined in $T_{c}\Lambda M\times T_{c}\Lambda M$ or in $T_c\Omega_{p}M\times T_c\Omega_{p}M$,  where  $\Omega_{p}M:=\Omega_{pp}M$.

Let us call $\lambda(\gamma)$ the index of $I_{\gamma}$,  $\tilde \lambda(c)$ the index of $I_c$ and $\lambda(c)$ the index of the restriction of $I_c$ to the space $T_{c}\Omega_{p}M\times T_{c}\Omega_{p}M$.

Clearly, $T_c \Omega_{p}M\subset T_c \Lambda M$ and consequently
$\tilde \lambda(c)\geq \lambda(c)$. The difference between the two indices is the so-called {\em concavity index} $\mathrm{con} (c)$. 

It is well known that $\mathrm{con} (c)$ is uniformly bounded (see Equation (1.4) in \cite{BaThZi82}). For the reader convenience, we will now prove this property in the Finslerian setting.
The following lemma  in \cite{BiMePi08} is an infinite dimensional version of a result in \cite[page 120]{Art57}.
\begin{lemma}\label{BTZ}
Let $X$ be a Hilbert space and let $B$ be a continuous symmetric essentially positive bilinear form on $X$. If $W\subset X$ is a closed subspace and $S$ denotes the $B$-orthogonal space to $W$, then:
\[n_-(B)=n_-\left(B|_{W\times W}\right)+n_-\left(B|_{S\times S}\right)+\dim (W\cap S)-
\dim (W\cap\ker (B))\]
\end{lemma}
 
 Let ${\mathds J}(c)$ be the vector space of Jacobi fields along $c$. 
 We will consider the following subspaces: 
 \begin{align*}
 \mathds J^{cl}(c)&=\{J\in{\mathds J}(c):J(0)=J(1)\},\\
 \mathds J^p(c)&=\{J\in{\mathds J}(c):J(0)=J(1), J'(0)=J'(1)\},\\
 \mathds J^0(c)&=\{J\in{\mathds J}(c):J(0)=J(1)=0\}.
 \end{align*}
 Observe that the kernel of
 $I_c$ is equal to $\mathds J^p(c)$ and the kernel of $I_c$ restricted  to the space $T_{c}\Omega_{p}M\times T_{c}\Omega_{p}M$ is  $\mathds J^0(c)$.
 Moreover, 
 the $I_c$-orthogonal subspace to $T_c\Omega_{p}(M)$ is equal to $\mathds J^{cl}(c)$.
 \begin{proposition}\label{indices}
Let us define the bilinear form $b:\mathds J^{cl}(c)\times \mathds J^{cl}(c)
 \rightarrow \R$ as 
 \[b(J_1,J_2)=g_T(J'_1(1)-J'_1(0),J_2(0)).\]  
  Then
 \[\tilde{\lambda}(c)=\lambda(c)+\dim \mathds J^0(c)-\dim\big(\mathds J^0(c)\cap \mathds J^p(c)\big)+n_-(b).\]
 \end{proposition}
 \begin{proof}
 That the bilinear form $I_c$ is continuous, symmetric and essentially positive can be shown as in Lemma 2 of \cite{CaJaMa09}. Moreover,   applying integration by parts we deduce that $I_c|_{\mathds J^{cl}(c)\times \mathds J^{cl}(c)}=b$. Then the proof is a direct application of Lemma \ref{BTZ} with $B=I_c$, $X=T_c\Lambda M$ and $W=T_c\Omega_pM$.
 \end{proof}
 Observe that $\ker(b)=\mathds J^0(c)+ \mathds J^p(c)$. Then 
 from Proposition \ref{indices}, it follows that 
 \begin{equation}
 \label{conc}
 0\leq\mathrm{con}(c)=\dim \ker(b)+n_-(b)-\dim \mathds J^p(c)\leq 2\dim M.
 \end{equation} 
Let us remark that the upper bound in the last inequality can be improved until
$\dim M-1$ as in \cite[Eq. (1.5)]{BaThZi82}.

Let  $p$ and $q$ be two points in $M$ such that there 
exists a prime closed geodesic $c$ passing through them. We call $\gamma^0$ the geodesic arc of $c$ from $p$ to $q$. Choosing a suitable affine parametrization of $c$, denoted by $c_q$, such that $c(0)=c(1)=q$, we call $\gamma^m$ the geodesic composed by $\gamma^0$ and $c^m_q$, where $c^m_q(s)=c_q(m s)$ is the $m$-th iteration of $c_q$.
\begin{lemma}\label{arcs}
Let $(M,F)$ be a Finsler manifold and $p,r,q$ three points in $M$ such that there
exists a geodesic $\gamma\colon[0,1]\to M$ from $p$ to $q$ passing through $r$. Let $\gamma_1$ be the arc of $\gamma$ from $p$ to $r$ and $\gamma_2$ the one from $r$ to $q$ (both parametrized on $[0,1]$). Then
\begin{equation}\label{desind}
\lambda (\gamma)\geq \lambda(\gamma_1)+\lambda(\gamma_2).
\end{equation}
\end{lemma}
\begin{proof}
Consider the linear maps 
\baln
&V_1\in T_{\gamma_{1}}\Omega_{pr}M\longmapsto W_1\in T_{\gamma}\Omega_{pq}M,&&V_2\in T_{\gamma_2}\Omega_{rq}M\longmapsto W_2\in T_{\gamma}\Omega_{pq}M,\\
&W_1=\begin{cases} V_1(s/t_0)&s\in [0,t_0],\\
		  0&s\in (t_0, 1],
    \end{cases}&&
W_2=\begin{cases} 0&s\in [0,t_0],\\
V_2\left(\dfrac{s-t_0}{1-t_0}\right)&s\in (t_0,1].
       \end{cases}
\ealn
They are injective maps and their  images constitute two orthogonal subspaces with respect to the scalar product $\langle V,W\rangle=\inte g_T(V',W')\de s$ in $T_{\gamma}\Omega_{pq}M$.
As
\baln 
I_{\gamma}(W_1,W_1)=\frac{1}{t_0}I_{\gamma_1}(V_1,V_1)&&\text{and} && I_{\gamma}(W_2,W_2)=\frac{1}{1-t_0}I_{\gamma_2}(V_2,V_2),
\ealn
\eqref{desind} follows.
\end{proof}
We observe that Lemma 1 in \cite{GroMey69a} about the Morse index of iterated Riemannian closed geodesics can be reproduced word for word in Finsler closed geodesics  (cf. also last section in \cite{Radema89}).
\bl\label{gmlemma}
Let $c\colon S^1\to M$ be a closed geodesic of $(M,F)$ and $m\in \N\setminus\{0\}$.  Consider the $m$-th iterate $c^m(s)=c(m s)$ of $c$. Then either 
$\tilde \lambda(c^m)=0$ for all $m$ or there exist $a_1, a_2>0$ such that for each $m>m'$ 
\beq\label{gm}
\tilde \lambda (c^m)-\tilde \lambda (c^{m'})\geq a_1(m-m')- a_2.
\eeq
\el
\begin{remark}\label{gammam0}
Observe that if $\tilde\lambda(c^m)=0$ for any $m\in\N$, then also $\lambda (\gamma^{m-1})=0$ for all $m\in\N$. Indeed, first notice that $\lambda(c_r^m)=0$
for all $r\in c(S^1)$. Let $\gamma_{qp}$ be the geodesic arc from $q$ to $p$ such that $c_p$ is the composition of $\gamma^0$ and $\gamma_{qp}$. From Lemma~\ref{arcs}, 
\[0=\lambda (c^{m}_p)\geq \lambda(\gamma^{m-1})+\lambda (\gamma_{qp}),\] 
hence $\lambda(\gamma^{m-1})=0$.
\end{remark}
An analogous result to Lemma \ref{gmlemma} holds also for the iterated geodesics between two points. 
\bl\label{GM2}
Let $c\colon S^1\to M$ be a closed geodesic passing through the points $p$ and $q$. For $m\in \N\setminus\{0\}$, let $\gamma^m= \gamma^0+ c^m_q$ then
either  $\lambda(\gamma^m)=0$ for all $m\in\N$ or there exist $b_1, b_2>0$ such that for each $m>m'+1$
\beq\label{gm2}
\lambda (\gamma^m)-\lambda (\gamma^{m'})\geq b_1(m-m')- b_2.
\eeq
\el
\begin{proof}
Observe that, by using Lemma~\ref{arcs}, we deduce that
\begin{align}
&\lambda(\gamma^m)\geq \lambda(c_p^m)+\lambda(\gamma_{pq}),\nonumber\\
&\lambda(c_p^{m'+1})\geq \lambda(\gamma^{m'})+\lambda(\gamma_{qp}),\label{mprimouno}
\end{align}
where $\gamma_{pq}$ and $\gamma_{qp}$ are the arcs of $c$ that go from $p$ to $q$ and from $q$ to $p$ respectively. Then summing the inequalities in
\eqref{mprimouno},
\begin{align}%\label{dessuperlinear}
\lambda(\gamma^m)-\lambda(\gamma^{m'})&\geq \lambda(c_p^m)-\lambda(c_p^{m'+1})+
\lambda(\gamma_{pq})+\lambda(\gamma_{qp}),\nonumber
\end{align}
and recalling that $\tilde\lambda(c^m)=\lambda(c^m_p)+\mathrm{con}(c^m)$ and inequality \eqref{conc},
\begin{align}\label{dessuperlinear}
\lambda(\gamma^m)-\lambda(\gamma^{m'})&\geq \tilde{\lambda}(c^m)-\tilde{\lambda}(c^{m'+1})-\mathrm{con}(c^m)+\mathrm{con}(c^{m'+1})+
\lambda(\gamma_{pq})+\lambda(\gamma_{qp})\nonumber\\
&\geq \tilde{\lambda}(c^m)-\tilde{\lambda}(c^{m'+1})-2\dim M+
\lambda(\gamma_{pq})+\lambda(\gamma_{qp}).
\end{align}
 By Remark \ref{gammam0}, if there exists $n_0\in\N$ such that $\lambda(\gamma^{n_0})\not=0$, then   $\tilde\lambda(c^{n_0+1})\not=0$.  
  Thus we get \eqref{gm2} from \eqref{gm} and \eqref{dessuperlinear}.
\end{proof}
\begin{proof}[Proof of Theorem~\ref{teo}]
Let $N_k(p,q)$ be the number of geodesics between $p$ and $q$ having index $k$.  As $\Omega_{pq}M$ is homotopically equivalent to $\Omega M$ (that can be seen, for example, as in \cite[Th 1.2.10]{Kli}), we can assume that the sequence of Betti numbers $\beta_k(\Omega_{pq} M,\K)$ is unbounded.  By the Morse relations  (for the Riemannian case see  Corollary (3) at p. 338 of \cite{Palais63} and the subsequent remark, for the Finsler extension, see Theorem 9 of \cite{CaJaMa09}), we know  that $\beta_k(\Omega_{pq} M,\K)\leq N_k(p,q)$.
By contradiction, assume that there exist $(c_j)_{j=1,\ldots,l}$ prime closed geodesics covering each geodesic arc between $p$ and $q$. If, for all $j$ and $m$, 
$\lambda(\gamma^m_j)=0$,  the Morse relations give immediately a contradiction. Otherwise there exist $\bar m_1, \dots, \bar m_h$, $h\leq l$, such that $\lambda(\gamma^{\bar m_i}_i)\neq 0$ and from Lemma~\ref{GM2} there exist positive constants $b_{i1}, b_{i2}$ such that $\lambda (\gamma^m_i)-\lambda (\gamma^{m'}_i)\geq b_{i1}(m-m')- b_{i2}$. Hence we deduce that, for each $i$,  denoted by  $\gamma^{m'}_i$ the first iterate to have index $k$,  the number of subsequent iterates that may have index $k$ is less than $\max\{\frac{b_{i2}}{b_{i1}},2\}$, thus it is independent of $k$. Hence $N_k(p,q)$ and then also $\beta_k(\Omega_{pq} M,\K)$ are uniformly bounded with respect to $k$ by  $K=h \cdot \max_{i=1,\ldots,h}\{\frac{b_{i2}}{b_{i1}},2\}$, getting a contradiction.
\end{proof}
 
\section*{Acknowledgment}
We became aware of the work by M. Tanaka thanks to the comments to the first version of this note by an anonymous  referee that we would like to acknowledge.

\end{document}